\documentclass[10pt]{article} 
\usepackage[dvips]{graphics} 
\usepackage[usenames]{color}
\usepackage{amsmath, amssymb, fancybox, fancyhdr} 
\usepackage{fig4tex} 
\usepackage{theorem}
\usepackage{amssymb, amsmath, amsbsy, amsfonts}
\usepackage{stmaryrd}

\usepackage{graphics,graphicx}
\usepackage{amssymb}
\usepackage{graphicx}
\usepackage{amsmath}
\usepackage{stmaryrd}

\newtheorem{theorem}{Theorem}[section] 

\newtheorem{lemma}[theorem]{Lemma} 
\newtheorem{proposition}[theorem]{Proposition} 
 
\newtheorem{corollary}[theorem]{Corollary} 

{\theorembodyfont{\rmfamily}

\newtheorem{remark}[theorem]{Remark}
}

\newenvironment{proof}{\noindent\text{\textbf{Proof.\:}}}{}

\newcommand{\N}{\mathbb{N}}
\newcommand{\Z}{\mathbb{Z}}
\newcommand{\R}{\mathbb{R}}
\newcommand{\T}{\mathbb{T}}
\renewcommand{\S}{\mathbb{S}}
\newcommand{\C}{\mathcal{C}}
\newcommand{\Om}{\Omega}

\def\Om{\Omega}

\def\te{\theta}

\def\qed{\hfill $\square$ \goodbreak \smallskip}
\def\eps{\varepsilon}

\title{On the local minimizers \\ of the Mahler volume}

\author{E.~M. Harrell\\
School of Mathematics\\ Georgia Institute of Technology\\
Atlanta, GA 30332-0160\\
{\tt harrell@math.gatech.edu}
\and
A. Henrot\\
Institut \'Elie Cartan Nancy\\
UMR 7502,
Nancy Universit\'e - CNRS - INRIA\\
B.P. 239 54506 Vandoeuvre les Nancy Cedex,  France\\
{\tt henrot@iecn.u-nancy.fr}
\and
 J. Lamboley\\
CEREMADE\\ Universit\'e Paris-Dauphine\\ Place du Mar\'echal de Lattre de Tassigny\\
 75775 Paris, France}

\date{\today}

\begin{document}
\maketitle
\begin{abstract}

We focus on the analysis of local minimizers of the Mahler volume, that is to say the local solutions to the problem
$$\min\{ M(K):=|K||K^\circ|\;/\;K\subset\R^d\textrm{ open and convex}, K=-K\},
$$
where $K^\circ:=\{\xi\in\R^d ; \forall x\in K, x\cdot\xi<1\}$ is the polar body of $K$, and $|\cdot|$ denotes the volume in $\R^d$.
According to a famous conjecture of Mahler the cube 
is expected to be a global minimizer for this problem.

{In this paper we express} the Mahler volume in terms of the support functional of the convex body, which allows us to compute first and second 
derivatives of the obtained functional. We deduce from these {computations} a concavity property of the Mahler volume which seems to be new. As a consequence of this property, we retrieve a result which supports the conjecture, namely that any local minimizer 
has a Gauss curvature that vanishes at any point where it is defined (first proven by Reisner, Sch\"utt and Werner in 2012, see \cite{RSW}). 
Going more deeply into the analysis in the two-dimensional case, we generalize the concavity property of the Mahler volume and also deduce a new proof that any local minimizer must be a parallelogram (proven by B\"or\"oczky, Makai, Meyer, Reisner in 2013, see \cite{BMMR}).
\\

{\it Keywords:\,} 
Shape optimization, convex geometry, Mahler conjecture.
\smallskip

\end{abstract}

\section{Introduction and results}

This paper is devoted to the analysis of local minimizers of the Mahler-volume functional. 
In particular we point out a concavity property of this functional, 
which supports the usual expectations about the minimizers, according to a well-known conjecture of Mahler. We also use this property to give new proofs of two results that can be found in \cite{RSW} and \cite{BMMR}, see Theorems \ref{th:dimd} and \ref{th:dim2} below.\\

\noindent{\bf Notation:}\\

Let $K\subset \R^d$ be a convex body, 
that is, $K$ is nonempty, open, convex, and bounded.
We can define the polar dual body of $K$:
$$K^\circ:=\left\{\;\xi\in\R^d \;/ \;\forall x\in K,\; x\cdot\xi<1\;\right\}.$$ 
The polar dual is always another convex body.  The {\it Mahler volume} $M(K)$ of 
$K$ is defined as the product of the volumes of $K$ and its polar dual:
$$M(K):=|K||K^\circ|,$$
where $|\cdot|$ denotes the volume in $\R^d$.\\

We say that $K$ is symmetric when $K$ is centrally symmetric, that is $-K=K$. In that case, one can interpret $K$ as the open unit ball of a 
norm on $\R^d$, for which $K^\circ$ is simply the open unit ball for the dual norm, and is also symmetric.\\

It is well known that the ball maximizes the Mahler volume among convex symmetric bodies: this is the Blaschke-Santal\'o inequality:
$$\forall K\textrm{ convex symmetric body, } M(K)\leq M(B^d),$$
where $B^d$ is the unit Euclidean ball, with equality if and only if $K$ is an ellipsoid.\\

The corresponding minimization problem 
is the subject of a notorious
and difficult conjecture, which is the main motivation for this paper. Let us recall this conjecture:\\

\noindent{\bf The symmetric Mahler's conjecture:}\\

The symmetric version of Mahler's conjecture asserts that for all convex symmetric bodies $K \subset \R^d$, we should have

\begin{equation}\label{eq:mahlersym}
M(Q^d) = M(O_d) \leq M(K),
\end{equation}
where $Q^d$ is the unit cube and $O_{d}=(Q^d)^\circ=\{x\in\R^d\;/\;\sum_{i}|x_{i}|<1\}$ is the unit octahedron.\\
In \cite{BM} a great step is achieved since the authors proved that the Mahler conjecture
is true up to an exponential factor. In \cite{K}, 
Kuperberg improved the constant in this exponential factor, by proving:
$$\forall K\textrm{ symmetric convex body of }\R^d,\; (\pi/4)^{d-1} M(Q^d)\leq M(K) .$$
The proof of Mahler's conjecture in dimension 2 appeared 
already in the original paper of Mahler \cite{M}. It has also been proved by Reisner in \cite{R} that equality in \eqref{eq:mahlersym} is 
attained only for parallelograms.
In higher dimensions, \eqref{eq:mahlersym} is still an open question.

It should be remarked at this point that $Q^d$ and $O_{d}=(Q^d)^\circ$ 
are not the only expected minimizers.
In the first place, the Mahler volume is not only invariant by duality ($M(K)=M(K^\circ)$) but is also an affine invariant, 
in the sense that 
if $T:\R^d\to\R^d$ is a linear invertible transformation, then $M(T(K))=M(K)$. 
Secondly, there is a sort of invariance with dimension: 
$$M(Q^{d_{1}}\times O_{d_{2}})=\frac{M(Q^{d_{1}})M(O_{d_{2}})}{\left(\begin{array}{c}d_{1}+d_{2}\\d_{1}\end{array}\right)}=M(Q^{d_{1}+d_{2}}),$$
since $M(Q^d)=\frac{4^d}{d!}$. Therefore products of cubes and octahedra, polar bodies of products of cubes and 
octahedra, etc. should also be minimizers. However, in dimension 2 and 3, the cube and the octahedra are the only expected minimizers, up to affine transformation.
See \cite{T} for 
further details and remarks.\\

\noindent{\bf The nonsymmetric Mahler's conjecture:}\\

One can also pose a nonsymmetric version of the Mahler conjecture, which is perhaps easier, because we expect the minimizer to be unique, 
up to invertible affine transformations that preserve a certain choice of the origin. 
To be precise, we introduce a more suitable version of the Mahler volume in this nonsymmetric setting: since $M(K)$ is not invariant by translation, one can choose the position of $K$ so that it minimizes the Mahler volume:
$$\mathcal P(K)=\inf\{|K||(K-z)^\circ|, z\in int(K)\}$$
the minimum being attained at a unique $z=s(K)$, known as the Santal\'o point of $K$ (see \cite{Sa}). We refer to this functional as the {\it nonsymmetric Mahler volume}.\\

Denoting by $\Delta_{d}$ a $d$-dimensional simplex, it is conjectured 
(see \cite{M,KR})
that for every convex body $K\subset\R^d$ containing 0,
$$\mathcal{P}(\Delta_{d})\leq \mathcal{P}(K), $$
with equality only if and only if $K$ is a $d$-dimensional simplex.\\

In this article we emphasize some concavity properties of the Mahler functional; these properties are contained in Sections \ref{ssect:concdimd} and \ref{ssect:concdim2}. First, we point out that these properties can be used to obtain some geometrical information
about the minimizers of the Mahler volume.
Namely, we provide new proofs of the following results (see \cite{RSW,BMMR}):

\begin{theorem}\label{th:dimd}
Let $K^*$ be a symmetric convex body in $\R^d$, which minimizes 
the Mahler volume among symmetric convex bodies:
\begin{equation}\label{eq:min}
M(K^*)=\min\left\{\;M(K): \; K {\rm \, convex \,\; symmetric \,\; body }\right\}.
\end{equation}
If $\partial K^*$ contains
a relatively open set $\omega$ of class  $\C^2$, then the Gauss curvature of $K^*$ vanishes on $\omega$.

The same result holds if $K^*$ is no longer symmetric, and is a minimizer of the nonsymmetric Mahler volume among convex bodies containing $0$:
\begin{equation}\label{eq:min2}
\mathcal{P}(K^*)=\min\{\mathcal{P}(K): {\rm \, convex \, body }\ni 0\}.
\end{equation}

\end{theorem}

\begin{theorem}\label{th:dim2}
In dimension $2$, any symmetric {\bf local} minimizer of \eqref{eq:min} is a parallelogram.
\end{theorem}

\begin{remark}
The word ``local'' in this result is to be understood in the sense of the $H^1$-distance between the support functions of the bodies; refer to Remark \ref{rem:local} for more details. Theorem \ref{th:dimd} is also valid for local minimizer, but here one can even restrict to neighbors of $K^*$ in the $\C^\infty$-topology, the boundary of which differ from $\partial K^*$ only in $\omega$.
\end{remark}

We briefly notice that in dimension higher than 2, an easy consequence of Theorem \ref{th:dimd} can be obtained on the regularity of minimizers:

\begin{corollary}
If $K^*$ is a minimizer for \eqref{eq:min} or \eqref{eq:min2}, then $K^*$ cannot be globally $\C^2$.
\end{corollary}

As we noticed before, in order to obtain Theorem \ref{th:dimd}, we apply the framework of calculus of variations
to a formulation of the Mahler volume 
in terms of the support function of the convex body, and we 
observe that the Mahler functional 
enjoys a certain concavity property, using the second-order derivative of the functional, see Propositions \ref{prop:conc} and Corollary \ref{cor:conc} (these ideas are inspired by some results in \cite{LN,BFL}). Therefore, it becomes quite natural that the minimizers should saturate the constraint, which is here the convexity of the shape. This result bolsters the conjecture and the intuition that the minimizers should contain flat parts, and that the Mahler volume should capture the roundness of a convex body.

Theorem \ref{th:dimd} can be considered as a local version of a 
result in \cite{S}, which asserts that if $K$ belongs to the class $\C^2_{+}$, that is to say $K$ is globally $\C^2$ and has a positive Gauss curvature everywhere on its boundary, then one can find a suitable deformation that decreases the Mahler volume (and preserves the symmetry if $K$ is itself symmetric).\\
Local minimality of the cube and the simplex
are also proved in \cite{NPRZ} and \cite{KR},
respectively for problem \eqref{eq:min} and \eqref{eq:min2}.

Even though the Mahler conjecture is still a distant hope, 
we emphasize that our analysis can be strengthened in dimension 2. More precisely, we can obtain in dimension 2 a more general concavity property of the Mahler volume (valid at any convex body, without assumption of regularity or positive curvature), see Proposition \ref{prop:conc2} and Corollary \ref{cor:conc2}, and then going deeper in the computations we obtain a new proof of Theorem \ref{th:dim2}.

We note that even though Theorems \ref{th:dimd} and \ref{th:dim2} are not new in the literature, our strategy is self-contained 
{and provides} 
short and new proofs, and we believe the concavity properties of the Mahler functional we point out are interesting on their own. Moreover, they underline a new challenge related to the Mahler conjecture, namely to express some concavity property of the Mahler volume that generalizes Corollary \ref{cor:conc} in the non-smooth setting, which we were only able to do in dimension 2 in Corollary \ref{cor:conc2} (see also \cite{MR} for additional concavity properties).

In the next section, we describe in which sense the Mahler volume enjoys a concavity property for $d$-dimensional shapes and deduce Theorem \ref{th:dimd}. In the third section we focus on the 2-dimensional improvements.

\section{The $d-$dimensional case}\label{sect:dimd}

\subsection{A functional formulation of the Mahler volume}\label{ssect:dimdformulation}

If $K$ is a convex body, one can define its support function $h_{K}:\S^{d-1}\to\R$ by:
$$h_{K}(\theta):=\sup \left\{\;x\cdot\theta,\; x\in K\;\right\}.$$
It is well-known (see \cite{Sc} for details, especially sections 1.7 and 2.5) that $h_{K}$ characterizes the convex body $K$, and that its positive 1-homogeneous extension $\tilde{h}_{K}$ to $\R^d$ (that is to say $\tilde{h}_{K}(\lambda x)=\lambda h_{K}(x),\;\forall \lambda\geq 0, \forall x\in\S^{d-1}$) is convex. In that case we shall say that $h_{K}$ is {\it convex}. Moreover, any functional $h:\S^{d-1}\to\R$,
the extension of which is convex, 
is the support function of a convex body. The volumes of $K$ and $K^\circ$ are conveniently written in
terms of $h_K$:
$$|K|=\frac{1}{d}\int_{\S^{d-1}}h_{K} \det(h_{K}''+h_{K}Id)d\sigma(\theta)\quad\textrm{ and }|K^\circ|=\frac{1}{d}\int_{\S^{d-1}}h_{K}(\theta)^{-d}d\sigma(\theta),$$
where $h_{K}''$ denotes the matrix of second covariant derivatives with respect to an orthonormal frame on $\S^{d-1}$, when it is well-defined (see the remarks below).
Hence the problem of minimizing the Mahler volume can formally be formulated as:

$$\min\left\{\frac{1}{d}\int_{\S^{d-1}}h \det(h''+hId)d\sigma\frac{1}{d}\int_{\S^{d-1}}h^{-d}d\sigma , h:\S^{d-1}\to\R \textrm{ convex}\right\}.$$

\noindent
In order to incorporate the symmetry constraint, one 
simply demands that
admissible $h$ be even.\\

These formulas are only valid if one can 
make sense of
$\det(h''+hId)$, which is not clear without regularity (one should use the 
surface area measure of Alexandrov \cite{Sc}).
Furthermore, some care is necessary in using the support function, because it is 
defined on the Gauss sphere, which is only a one-to-one image of $\partial K$ in the 
smooth, strictly convex case. Therefore we are going to localize the above formulation thanks to the following Lemma:

\begin{lemma}\label{lem:curv}
If $K$ is a convex body, and $\omega\subset\partial K$ such that $\omega$ is $\C^2$ and the Gauss curvature is positive on $\omega$, then 
$\Omega:=\nu_{K}(\omega)$, where $\nu_{K}$ is the Gauss map of the body $K$, is a nonempty open set in $\S^{d-1}$, $h_{K}$ is $\C^2$ in $\Omega$, and $det(h_{K}''+h_{K}Id)>0$ on $\Omega$.
\end{lemma}

\begin{proof} This is classical, and generally stated for convex bodies which are globally $\C^2$ and with a Gauss curvature everywhere positive, but the proof is actually local and so our lemma follows with usual arguments, see for example \cite[Section 2.5, p.~106]{Sc}.
\qed
\end{proof}

With the help of this Lemma, we can localize the optimization problem: we introduce $K^*$ a (local) solution of \eqref{eq:min} or \eqref{eq:min2} and $\omega\subset \partial K^*$ a relative open set, assumed to be $\C^2$ with positive Gauss curvature; we restrict the calculation to a relatively open subset $U\Subset\Omega:=\nu_{K^*}(\omega)$. Then considering shapes that differ from $K^*$ only on $\omega$, we easily deduce that $h_{0}:=h_{K^*}$ is a (local) solution of the following problem of calculus of variations:

\begin{equation}\label{eq:minh}
J(h_{0})=\min\{J(h), h:\S^{d-1}\to\R \textrm{ convex, }\C^2 \textrm{ in }\overline{U}\textrm{ and even}\},
\end{equation}
or, in the nonsymmetric case:
\begin{equation}\label{eq:minh2}
J(h_{0})=\min\{J(h), h:\S^{d-1}\to\R \textrm{ convex, }\C^2 \textrm{ in }\overline{U}\textrm{ and positive}\},
\end{equation}

\noindent
where

\begin{equation}\label{eq:J}
J(h)=J_{h_{0},U}(h):=A(h)B(h), \;\;\;\;A(h):=\frac{1}{d}\int_{U}h \det(h''+hId)d\sigma\quad\textrm{ and }B(h)=\frac{1}{d}\int_{U}h^{-d}d\sigma.
\end{equation}

\noindent
and we know that $h_{0}$ is $\C^2$ in $\Omega$ and $\det({h_{0}}''+h_{0}Id)>0$ on $\Omega$. Note that $J$ depends on $h_{0}$ and $U$ and can be considered as a localization of the Mahler volume.\\
Note that an analytical characterization of convexity in terms of second-order derivatives in this context is:
\begin{equation}\label{eq:conv}
\textrm{If the eigenvalues of }(h''+h Id)\textrm{ are nonnegative, then }h\textrm{ is convex.}
\end{equation}
\begin{remark}\label{rem:santalo}In \eqref{eq:minh2}, we drop the translation operation by the Santal\'o point, since this is an artificial constraint: a local minimizer among convex sets is also a local minimizer among sets whose Santal\'o point is zero, and reciprocally.
\end{remark}

\subsection{Concavity of the functional}\label{ssect:concdimd}

We prove here an estimate on the second order derivative of $J$, which implies a ``local'' concavity property of the functional, see \cite{LN,BFL} for similar results.\\

\begin{proposition}\label{prop:conc}
Let $\Omega\subset\S^{d-1}$ and $h_{0}:\S^{d-1}\to\R$ be of class $\C^2$ in $\Omega$, with $det(h_{0}''+h_{0}Id)>0$ on $\Omega$, and $U\Subset\Omega$ relatively open. Then $J$ admits second order directional derivative at $h_{0}$ in every direction $v\in\C^\infty_{c}(U)$, and moreover there exist $C=C(h_{0},U)$, and $\alpha=\alpha(h_{0},U)>0$ such that:
$$\forall v\in\C^\infty_{c}(U), \quad J''(h_{0})\cdot(v,v)\leq C\|v\|_{L^2(\Omega)}^2-\alpha|v|_{H^1_{0}(\Omega)}^2,$$
where $|v|_{H^1_{0}(\Omega)}^2=\int_{\Omega}|\nabla v|^2d\sigma$ denotes the standard 
norm on $H^1_{0}(\Omega)$ and $J=J_{h_{0},U}$ was introduced in \eqref{eq:J}.
\end{proposition}
As a direct consequence, we obtain the following Corollary which explains in what sense the previous result can be considered as a concavity property of the Mahler functional:
\begin{corollary}\label{cor:conc}
With the same notation as in Proposition \ref{prop:conc}, if $U'\Subset U$ is relatively open, and $\lambda_{1}^D(\Delta_{\S^{d-1}},U')>C/\alpha$, then
$$\forall v\in\C^\infty_{c}(U')\setminus\{0\}, \quad J''(h_{0})\cdot(v,v)<0,$$
where $\lambda_{1}^D(\Delta_{\S^{d-1}},U')=\inf\{|\varphi|_{H^1_{0}(U')}^2, \varphi\in \C^\infty_{c}(U'), \|\varphi\|_{L^2(U')}=1\}$ is the first Dirichlet eigenvalue of the Laplace operator on $U'\subset\S^{d-1}$.
\end{corollary}
\begin{remark}
The assumption $\lambda_{1}^D(\Delta_{\S^{d-1}},U')>C/\alpha$ is satisfied as soon as $U'$ has a small $(d-1)$-volume.
\end{remark}
\begin{proof}
The existence of directional derivatives is easy with the regularity assumed on $h_{0}$ and $v$.
Refering to \cite[Proposition 5,6]{C} for more detailed calculations, we get
\begin{eqnarray*}
\forall v\in \C^\infty_{c}(U), &&A'(h_{0})\cdot v=\int_{\S^{d-1}}v \det(h_{0}''+h_{0}Id)d\sigma,\\
&&A''(h_{0})\cdot(v,v)=\int_{\S^{d-1}}v\sum_{i,j\leq d-1}c_{ij}(\partial_{ij}v+\delta_{ij}v)d\sigma=\int_{\S^{d-1}}\left(Tr(c_{ij})v^2-\sum_{i,j\leq d-1}c_{ij}\partial_{i}v\partial_{j}v\right)d\sigma,
\end{eqnarray*}
where $(c_{ij})_{1\leq i,j\leq d-1}$ is the cofactor matrix of $(h_{0}''+h_{0}Id)=(\partial_{ij}h_{0}+h_{0}\delta_{ij})_{1\leq i,j\leq d-1}$, that is to say
$(c_{ij})_{1\leq i,j\leq d-1}=\det(h_{0}''+h_{0}Id)(h_{0}''+h_{0}Id)^{-1}$ (For the last formula, we integrate by parts and use Lemma 3 in \cite{C}).\\
Moreover, we easily get
$$
B'(h_{0})\cdot v= -\displaystyle\int_{\S^{d-1}}\frac{v}{h_{0}^{d+1}}d\sigma, \quad B''(h_{0})\cdot (v,v)= (d+1)\displaystyle{\int_{\S^{d-1}} \frac{v^2}{h_{0}^{d+2}}}d\sigma.
$$
Therefore

\begin{eqnarray*}
(AB)''(h_{0})\cdot (v,v)&=&A''(h_{0})\cdot (v,v)B(h_{0})+2A'(h_{0})\cdot (v) B'(h_{0})\cdot (v) +B''(h_{0})\cdot (v,v)A(h_{0})\\
&=&A(h_{0})(d+1)\displaystyle{\int_{\S^{d-1}} \frac{v^2}{h_{0}^{d+2}}}d\sigma-2\int_{\S^{d-1}}v \det(h_{0}''+h_{0}Id)d\sigma\displaystyle\int_{\S^{d-1}}\frac{v}{h_{0}^{d+1}}d\sigma\\
&&+B(h_{0})\int_{\S^{d-1}}\left(Tr(c_{ij})v^2-\sum_{i,j\leq d-1}c_{ij}\partial_{i}v\partial_{j}v\right)d\sigma,
\end{eqnarray*}
but the eigenvalues of the matrix $(c_{ij})$ are $\kappa_{i}/\kappa$, where $\kappa$ is the Gauss curvature and $\kappa_{i}$ are the principal curvatures \cite[Corollary 2.5.2]{Sc}.
Therefore
$$\sum_{i,j\leq d-1}c_{ij}\partial_{i}v\partial_{j}v\geq \beta |\nabla v|^2,$$
where $\beta(\theta)=\min_{i}\kappa_{i}(\theta)/\kappa(\theta)$.
This then leads to the result, with 
$$C=(d+1)A(h_{0})\left\|\frac1{h_{0}^{d+2}}\right\|_{L^\infty(U)}
+2\left\|\frac1\kappa\right\|_{L^\infty(U)}\left\|\frac1{h_{0}^{d+1}}\right\|_{L^\infty(U)}
\mathcal{H}^{d-1}(U)+B(h_{0})\left\|\frac H\kappa\right\|_{L^\infty(U)},$$
where $H=\sum_{i} \kappa_{i}$, and $\alpha=B(h_{0})\min_{\theta\in U}\beta(\theta)$, 
which is positive since $det(h_{0}''+h_{0}Id)>0$ on $\Om$.
\qed
\end{proof}

\subsection{Proof of Theorem \ref{th:dimd}}\label{sect:conc}
\noindent{\bf $\bullet$ Nonsymmetric case:}
Let $K^*$ {be a} local minimizer of the Mahler volume $\mathcal{P}$ such that $0\in\overset{\circ}{K}$. Then it is also a local minimizer of $M$, since the translation operation by the {Santal\'o} point is an artificial constraint when dealing with the minimization problem.

We assume there exists $\omega$ a $C^2$ subset of $\partial K^*$ where the Gauss curvature is greater than $\alpha>0$. 
Then with Lemma \ref{lem:curv}, $h_{0}=h_{K^*}$ is optimal for the following problem:
$$\min\left\{J(h):=\frac{1}{d}\int_{\Om}h \det(h''+hId)d\sigma\frac{1}{d}\int_{\Om}h^{-d}d\sigma , h:\S^{d-1}\to(0,\infty) \textrm{ convex}\right\},$$
where $\Om=\nu_{K^*}(\omega)$.

For all $v\in \C^\infty_{c}(\Om)$, $h_{0}+tv$ is still the support function of a convex set for sufficiently small $|t|$:
indeed, the eigenvalues of  $(h_{0}+tv)''+(h_{0}+tv)Id$ are nonnegative, since they are close to those of $h_{0}''+h_{0}$, which are positive, and we use \eqref{eq:conv}. Therefore 
the second-order optimality condition and Proposition \ref{prop:conc} yields
\begin{equation}\label{eq:ordre2}
\forall v\in\C^\infty_{c}(\Om), 0\leq J''(h_{0})\cdot(v,v)\leq  C\|v\|_{L^2(\Omega)}^2-\alpha|v|_{H^1_{0}(\Omega)}^2.
\end{equation}
This would imply the false imbedding $L^2(\Om)\subset H^1_{0}(\Om)$, which is a contradiction.\qed

\noindent{\bf $\bullet$ Symmetric case:}
A similar proof as for the nonsymmetric case applies: indeed, we just need to restrict ourselves to symmetric perturbations.
With the same notation as in the previous proof, we can assume without restriction, that $\Omega$ is symmetric in the sense that $-\Om=\Om$. $h_{0}$ is therefore solution of
$$\min\left\{J(h), h:\S^{d-1}\to(0,\infty) \textrm{ convex and even}\right\}.$$
Therefore \eqref{eq:ordre2} is satisfied for any $v\in C^\infty_{c}(\Om)$ even, which also provides a contradiction.
\qed

\section{The 2-dimensional case}

In this section, we focus on the case $d=2$. Compare to Section \ref{sect:dimd}, we take advantage of the 2-dimensional framework to retrieve similar results without assumption of regularity or strict convexity of shapes. We therefore obtain a more general concavity property of the Mahler functional: this prevent the localization procedure from Section \ref{ssect:dimdformulation} and first implies that a minimizer of the Mahler problem is a polygon. With a more thorough analysis, this actually implies that a local minimizer of \eqref{eq:min}  is a parallelogram, which contains the results of Mahler and Reisner \cite{M,R}, that is to say inequality \eqref{eq:mahlersym} with the 
case of equality.

\subsection{A functional formulation of the Mahler volume}

We express the functional in terms of the support function, and since we work in dimension 2, 
we are now able to write the Mahler volume without any regularity assumption.\\

Using polar coordinates, we regard $\theta$ as in $\mathbb{T}=\R/(2\pi\Z)$ rather than in $\S^1$, and therefore $h_{K}:\T\to\R$ is viewed as a $2\pi$-periodic function. Therefore,
$$M(K)=\frac{1}{2}\int_{\T}(h_{K}^2(\theta)-h_{K}'^2(\theta))d\theta\int_{\mathbb{T}}\frac{1}{2h_{K}^2(\theta)}d\theta,$$
and the convexity constraint on the set can be written $h_{K}''+h_{K}\geq 0$, in the sense of a periodic distribution on $\R$. This implies for example that $h_{K}\in W^{1,\infty}(\T)$. We are therefore interested in the following optimization problems:

\begin{equation}
J(h_{0})=\min\left\{\;J(h):=A(h)B(h), \;h''+h\geq 0 \textrm{ and } \;\forall\;\theta\in\T,\; h(\theta)=h(\theta+\pi)\;\right\},
\end{equation}
and, in the nonsymmetric case,
\begin{equation*}
J(h_{0})=\min\left\{\;J(h), \;h''+h\geq 0 \textrm{ and }h>0\;\right\},
\end{equation*}
with the same notation as in the previous section:

$$A(h)=\frac{1}{2}\int_{\mathbb{T}}(h^2-h'^2)d\theta, \quad B(h)=\int_{\mathbb{T}}\frac{1}{2h^2}d\theta.$$
Note that compare to Section \ref{ssect:dimdformulation}, we assume no regularity nor strict convexity on the minimizer, and the above formulation is global.

\subsection{Concavity of the functional}\label{ssect:concdim2}

We now prove a 2-dimensional version of Propositions \ref{prop:conc}, dropping the regularity assumption on $h$:\\

\begin{proposition}\label{prop:conc2}
If $h_{0}\in H^1(\T)$ such that $h_{0}>0$ and $h''+h\geq 0$, then $J:H^1(\T)\to\R$ is twice Fr\'echet differentiable around $h_{0}$ and there exists $C=C(h)$, and $\alpha=\alpha(h)>0$ such that:
$$\forall v\in H^1(\T), \quad J''(h)\cdot(v,v)\leq C\|v\|_{L^\infty(\T)}\|v\|_{L^1(\T)}-\alpha|v|_{H^1(\T)}^2.$$
\end{proposition}

\begin{corollary}\label{cor:conc2}
With the same notation as in Proposition \ref{prop:conc2}, if $a\in(0,2\pi)$ and 
$$\lambda([0,a]):=\inf\left\{\frac{|\varphi|_{H^1_{0}(0,a)}^2}{\|\varphi\|_{L^1(0,a)}\|\varphi\|_{L^\infty(0,a)}}, \varphi\in \C^\infty_{c}(0,a)\right\}>C/\alpha,$$
 then
$$\forall v\in H^1_{0}(0,a)\setminus\{0\}, \quad J''(h)\cdot(v,v)<0.$$
\end{corollary}
\begin{remark}
If $a$ is small enough, then the assumption $\lambda([0,a])>C/\alpha$ is satisfied. It would be interesting to investigate the best computation of $C$ and $\alpha$ in order to obtain the maximal $a$ satisfying such a property. In a way, this is the strategy of the Section \ref{ssect:proofdim2} for specific deformations $v$ combining the information given by the first order condition. 
\end{remark}
\begin{proof}
The regularity of $J$ is obvious. We also easily get:
$$
\begin{array}{ll}
A'(h)\cdot v= \displaystyle\int_{\T} hv-h'v', &A''(h)\cdot (v,v)= \displaystyle\int_{\T} v^2-v'^2,\\
B'(h)\cdot v= -\displaystyle\int_{\T}\frac{v}{h^3}, & B''(h)\cdot (v,v)= 3\displaystyle{\int_{\T} \frac{v^2}{h^4}},
\end{array}
$$
and so
\begin{eqnarray*}
\nonumber(AB)'(h)\cdot v&=& B(h)\displaystyle\int_{\T} (hv-h'v')-A(h)\displaystyle\int_{\T}\frac{v}{h^3},\\
(AB)''(h)\cdot (v,v)&=& B(h)\displaystyle\int_{\T} (v^2-v'^2)-2\displaystyle\int vd(h''+h) \displaystyle\int_{\T}\frac{v}{h^3}+3A(h)\label{eq:J''}\displaystyle{\int_{\T} \frac{v^2}{h^4}},
\end{eqnarray*}
where $h''+h$ is a nonnegative Radon measure on $\T$.
The following local concavity estimate follows:
$$(AB)''(h)\cdot (v,v)\leq \left(B(h)+3A(h)\|1/h^4\|_{L^\infty(\T)}\right)\|v\|_{L^2(\T)}^2+2(h''+h)(\T)\|1/h^3\|_{L^\infty(\T)}\|v\|_{L^\infty(\T)}\|v\|_{L^1(\T)}-B(h)|v|_{H^1(\T)}^2,$$
where $(h''+h)(\T)$ is the total mass of the measure $h''+h$. This leads to the result with $\alpha=B(h)>0$, since $\|v\|_{L^2(\T)}^2\leq\|v\|_{L^\infty(\T)}\|v\|_{L^1(\T)}$.\qed
\end{proof}
\begin{remark}
One can also conclude that 
$$\forall v\in\C^\infty(\T)\textrm{ such that }(AB)'(h)\cdot v =0, \quad(AB)''(h)\cdot(v,v)\leq C\|v\|_{L^2(\T)}^2-\alpha|v|_{H^1(\T)}^2,$$
since $(AB)'(h)\cdot v =0$ implies that the middle term $2(A'(h)\cdot v )(B'(h)\cdot v)$ is nonpositive.
\end{remark}

\subsection{Proof of Theorem \ref{th:dim2}}\label{ssect:proofdim2}

We now focus on the proof of Theorem \ref{th:dim2} about local minimizers of the Mahler-volume in $\R^2$.
In comparison with Theorem \ref{th:dimd} which follows directly from Proposition \ref{prop:conc}, Theorem \ref{th:dim2} is no longer an easy consequence of Proposition \ref{prop:conc2}, so we give a detailed proof here.\\

Let $K$ be a local minimizer of the Mahler volume among symmetric convex bodies. 
 So $h_{0}=h_{K}$ is solution of the following problem:

\begin{equation}\label{eq:mahlersym2d}
J(h_{0})=\min\left\{\;J(h):=A(h)B(h), \;h:\T\to(0,\infty),\textrm{ such that }\;h''+h\geq 0 \textrm{ and } \;\forall\;\theta\in\T,\; h(\theta)=h(\theta+\pi)\;\right\},
\end{equation}
with the same notation as in the previous section:
$$A(h)=\frac{1}{2}\int_{\mathbb{T}}(h^2-h'^2)d\theta, \quad B(h)=\int_{\mathbb{T}}\frac{1}{2h^2}d\theta.$$
\begin{remark}\label{rem:local}
By ``local'', we mean 
that $K$ is minimal among all convex sets whose support function is close to that 
of $K$ in the $H^1$-norm.
More precisely, we say that $K$ is a local minimizer of the Mahler volume among symmetric convex bodies if there exists $\eps>0$ such that
\begin{equation}\label{eq:sym}
\forall L\textrm{ convex symmetric body such that }\|h_{L}-h_{K}\|_{H^1(\T)}\leq\eps,\quad M(K)\leq M(L).
\end{equation}
Another useful distance is the Hausdorff distance, 
expressible through the support functions by $\|h_{L}-h_{K}\|_{L^\infty(\T)}$.
It is an easy consequence of the Poincar\'e inequality that
the Hausdorff distance is bounded above by the $H^1$-distance, up to an universal constant (see \cite{AW} for example).

The converse inequality is not clear, but one can prove that the convergence in the sense of Hausdorff implies the convergence in the $H^1$-distance, and so there is topological equivalence. We give a short sketch of proof of this last property:

if $h_{n},h_{\infty}$ are such that $h_{n}''+h_{n}\geq 0, h_{\infty}''+h_{\infty}\geq 0,$ and $h_{n}\to h_{\infty}$ in $L^\infty(\T)$, then it is easy to see that $h_{n}$ is bounded in $W^{1,\infty}(\T)$ by a constant $C$ (see for example \cite[Lemma 4.1]{LN}), and therefore that
$$\int_{\T}d|h_{n}''|\leq \int_{\T}d(h_{n}''+h_{n})+\int_{\T}d|h_{n}|\leq 2\int_{\T}d|h_{n}|\leq 2C.$$

\noindent
Therefore $h_{n}'$ is bounded in $BV(\T)$, so up to a subsequence, $h_{n}'\to h'$ a.e. and in $L^1(\T)$ (by the compact imbedding of $BV(\T)$ in $L^1$). We conclude with the dominated convergence theorem that $h_{n}\to h_{\infty}$ in $H^1(\T)$, and by uniqueness 
of the accumulation point of $h_{n}$ that the whole sequence converges.
\end{remark}

\noindent{\bf $\bullet$ First step: Any local minimal set is a polygon}

We cannot 
directly apply Theorem 2.1 from \cite{LN}, since our functional is not exactly of the type of the ones considered there, and also because the constraints are slightly different, but one can follow the same argument, as is done
in the following lines.\\

Assume for the purpose of contradiction that {$K^*$, a local minimizer, is not a polygon and let $h_{0}$ denote} its support function.  Then there must exist an accumulation point $\theta_0$ of
$\textrm{supp}(h_{0}''+h_{0})$.

Without loss of generality we may assume that
$\theta_0=0$ and also that there exists a decreasing sequence $(\eps_n)$ tending to $0$ such that
$\textrm{supp}(h_{0}''+h_{0})\cap(0,\eps_n)\neq \emptyset$. As in \cite{LN} we follow an idea of T. Lachand-Robert
and M.A. Peletier (see \cite{LP01New}): for any $n\in\N$, we choose $0<\eps_n^i<\eps_n$, $i\in\llbracket1,4\rrbracket$, increasing with respect to $i$, such that $\textrm{supp}(h_{0}''+h_{0})\cap(\eps_n^i,\eps_n^{i+1})\neq\emptyset$, $i=1,3$. We consider $v_{n,i}$ solving
\[
 v_{n,i}''+v_{n,i}=\chi_{(\eps_n^i,\eps_n^{i+1})}(h_0''+h_0),\quad v_{n,i}=0 \mbox{ in }
(0,\eps_n)^c,\; i=1,3.
\]
Such $v_{n,i}$ exist since 
$\eps_n^i$ have been chosen so as
to avoid the spectrum of the Laplace operator with Dirichlet boundary conditions.
Next, we look for $\lambda_{n,i},\ i=1,3$ such that
${\displaystyle v_n=\sum_{i=1,3}\lambda_{n,i} v_{n,i}}$ satisfy
\[
v'_n(0^+)=v'_n(\eps_n^-)=0.
\]

The above derivatives exist since $v_{n,i}$ are regular near $0$ and $\eps_n$ in $(0,\eps_n)$.
We can always find such $\lambda_{n,i}$, as they satisfy two linear equations.
This implies
that $v_n''$ does not have any Dirac mass at $0$ and $\eps_n$, and therefore, $h+tv_n$ is the support function of a convex set, for $|t|$ small enough 
($n$ now being fixed).
We define the symmetric version of $v_{n}$: 
\begin{equation}\label{eq:sym2d}
\widetilde{v_{n}}(\theta)=\left\{\begin{array}{l}
v_{n}(\theta)\textrm{ if }\theta\in (0,\pi),\\
v_{n}(-\theta)\textrm{ if }\theta\in (-\pi,0),\\
0\textrm{ otherwise.}
\end{array}
\right.
\end{equation}
Therefore $h+t\widetilde{v_{n}}$ is an admissible function for \eqref{eq:mahlersym2d}.

So the second-order optimality condition yields
\begin{equation*}
0\leq
J''(h_{0})\cdot(\widetilde{v_{n}},\widetilde{v_{n}})\leq C\|\widetilde{v_{n}}\|_{L^\infty(\T)}\|\widetilde{v_{n}}\|_{L^1(\T)}-\alpha|\widetilde{v_{n}}|_{H^1(\T)}^2\leq 2(C\eps_{n}^2-\alpha)|v_{n}|_{H^1(\T)}^2
\end{equation*}
using Proposition \ref{prop:conc2} and the Poincar\'e inequality 
$$\forall\; v\in H^1(\T)\textrm{ such that }\textrm{supp}(v)\subset[0,\eps],\forall x\in[0,\eps],\; |v(x)| \leq \sqrt{\eps}|v|_{H^1(\T),}$$ with
$\eps=\eps_n$.\\
As $\eps_n$ tends to $0$, this inequality becomes impossible, which
proves that $\textrm{supp}(h_{0}''+h_{0})$ has no accumulation points.
It follows that $h_0''+h_0$ is a sum of positive Dirac masses,
which is to say that $K^*$ is a polygon.
\begin{remark}
It is easy to see that a similar argument applies in the nonsymmetric case.
\end{remark}

\noindent{\bf $\bullet$ Step 2: Another expression for $B$ and its derivatives:}\\

Since $K$ is a polygon,
\begin{equation}\label{eq:h''+h}
   h_{0}''+h_{0}=\sum_{i=0}^{2N-1}a_i\delta_{\te_i} \textrm{ for some }N\in\N^*, \te_{i}\in \T\textrm{ and }a_{i}>0.
\end{equation}
 We want to prove that $K$ is a parallelogram, that is to say $N=2$ in \eqref{eq:h''+h}.

As in the previous step, we would like to find a perturbation $v$ such that $J''(h_{0})\cdot(v,v) <0$, which would be a contradiction. So that $h_{0}+tv$ remains admissible for all small $t$, we need $v''+v$ to be supported within the support of $h_{0}''+h_{0}$.

Again we shall symmetrize the perturbation $v\in H^1_{0}(0,\pi)$ with \eqref{eq:sym2d}, and we easily prove 
$J'(h_{0})\cdot \widetilde{v}=2J'(h_{0})\cdot v$ since $h_{0}$ is symmetric, and
\begin{equation}\label{eq:AB''sym}
J''(h_{0})\cdot(\tilde{v},\tilde{v})=2A''(h_{0})\cdot(v,v)B(h_{0})+8A'(h_{0})\cdot(v)B'(h_{0})\cdot(v)+2B''(h_{0})\cdot(v,v)A(h_{0}).
\end{equation}

Since the expression for $B$ is not very tractable from the
geometric point of view, 
we would like to rewrite $B$ and its derivatives when one knows that $h$ is the support function of a polygon, and that $v$ is a deformation such that $v''+v$ is supported within the 
discrete set on which
$h''+h$ is nonzero.

Let us denote by $A_i, i=0\ldots 2N-1$ the vertices of $K$. Then
the support function $h$ is defined by
\begin{equation*}
   h(\te)=\rho_i \cos(\te - \alpha_i) \quad\mbox{for } \te\in
   (\te_i,\te_{i+1}),
\end{equation*}
where $\rho_i=OA_i, \alpha_i=(\overrightarrow{e_1},
\overrightarrow{OA_i})$ and $(\te_i,\te_{i+1})$ are the two angles
of the normal vectors of sides adjacent to $A_i$. Therefore
\begin{eqnarray}\label{8}
   B(h)&=&\int_0^{2\pi} \frac{d\te}{2h^2(\te)}=\sum_{i=0}^{2N-1}
\int_{\te_i}^{\te_{i+1}} \frac{d\te}{2\rho_i^2 \cos^2(\te -
\alpha_i)}\nonumber\\
&=&\sum_{i=0}^{2N-1}
\frac{1}{2\rho_i^2}\left.\tan(\te-\alpha_i)\right|_{\te=\te_i}^{\te=\te_{i+1}}
=\sum_{i=0}^{2N-1} \frac{\sin(\te_{i+1}-\te_i)}{2h(\te_i)h(\te_{i+1})}.
\end{eqnarray}
Now, when we replace $h$ by $h+tv$ where $v''+v=\sum_{i}\beta_{i}\delta_{\te_{i}}$, the angles of the new polygon
are unchanged, 
because $(h+tv)''+(h+tv)$ is a sum of nonnegative Dirac masses at
the same points, when $t$ is small enough. Thus we can compute the first and second derivative
of $B(h)$ using formula (\ref{8}), obtaining:
\begin{equation*}
   B'(h)\cdot v=-\sum_{i=0}^{2N-1} \frac{\sin(\te_{i+1}-\te_i)}{2h(\te_i)h(\te_{i+1})}
   \left[\frac{v(\te_i)}{h(\te_i)}+\frac{v(\te_{i+1})}{h(\te_{i+1})}\right]=
   -\sum_{i=0}^{2N-1} \left[\frac{\sin(\te_{i+1}-\te_i)}{h(\te_{i+1})}+\frac{\sin(\te_{i}-\te_{i-1})}{h(\te_{i-1})}\right]
 \frac{v(\te_i)}{h^2(\te_i)}
\end{equation*}
and
\begin{equation}\label{eq:B''}
B''(h)\cdot(v,v)=\sum_{i=0}^{2N-1}
\frac{\sin(\te_{i+1}-\te_i)}{h(\te_i)h(\te_{i+1})}
   \left[\frac{v^2(\te_i)}{h^2(\te_i)}+\frac{v^2(\te_{i+1})}{h^2(\te_{i+1})}
   +\frac{v(\te_i)v(\te_{i+1})}{h(\te_i)h(\te_{i+1})}\right]\,.
\end{equation}
Therefore the first optimality condition becomes:
$A(h)B'(h)\cdot v+B(h)A'(h)\cdot v=0$ for any $v$ symmetric ({\it i.e.}, for any $v(\te_i), i\in\llbracket 0,N-1\rrbracket$), and we
get:
\begin{equation}\label{11}
B(h)a_i-\frac{A(h)}{2h^2(\te_i)}\,\left(\frac{\sin(\te_{i+1}-\te_i)}{h(\te_{i+1})}
+ \frac{\sin(\te_{i}-\te_{i-1})}{h(\te_{i-1})}\right)=0\quad
\mbox{for } i=0,\ldots N-1\,.
\end{equation}

\noindent{\bf$\bullet$ Step 3: Optimality conditions for a simple deformation:}\\

We choose $v$ such that $v''+v=\alpha\delta_{\te_{1}}$ and $v\in H^1_{0}(\te_{0},\te_{2})$.
Therefore equations \eqref{11}, \eqref{eq:J''}, \eqref{eq:AB''sym} and \eqref{eq:B''} give
\begin{eqnarray}
J''(h)\cdot(\tilde{v},\tilde{v})&=&2B(h)\displaystyle\int vd(v''+v)-8\frac{B(h)}{A(h)}\left(\displaystyle\int vd(h''+h)\right)^2+2A(h)
\left[\frac{\sin(\te_{2}-\te_1)}{h(\te_{2})}+\frac{\sin(\te_{1}-\te_0)}{h(\te_{0})}\right]
\frac{v^2(\te_1)}{h^3(\te_1)}\nonumber\\\nonumber
&=&2B(h)\alpha v(\te_{1})-8\frac{B(h)}{A(h)}(a_{1}v(\te_{1}))^2+2A(h)
\left[ \frac{2B(h)h(\te_{1})^2a_{1}}{A(h)}\right]
\frac{v^2(\te_1)}{h^3(\te_1)}\\\label{eq:J''bis}
&=&2B(h)\left[-\frac{\sin(\te_{2}-\te_{0})}{\sin(\te_{2}-\te_{1})sin(\te_{1}-\te_{0})}-4\frac{a_{1}^2}{A(h)}+2
\frac{a_{1}}{h(\te_{1})}\right]v^2(\te_{1}),
\end{eqnarray}
where the last equality is obtained because a straightforward calculation gives 
$\alpha=-\frac{\sin(\te_{2}-\te_{0})}{\sin(\te_{2}-\te_{1})\sin(\te_{1}-\te_{0})}v(\te_{1})$.\\

\noindent{\bf$\bullet$ Step 4: Conclusion}\\

Let us assume, for a contradiction, that $K$ has at least
6 sides. Let $\te_0, \te_1, \te_2$ be the three first angles
of the normal, in such a way that the support function of $K$ satisfies
\begin{equation*}
   h''+h=a_0\delta_{\te_0} + a_1\delta_{\te_1} + a_2\delta_{\te_2}
   + \ldots,
\end{equation*}
and $\te_{2}-\te_{0}<\pi$.\\
%
We recall that the Mahler functional is invariant by affine transformation. Therefore, if $K$ is a local minimizer, the image of $K$ by such a transformation $T$ remains a local minimizer of $J$, since the neighbors of $K$ are transformed in neighbors of $T(K)$ by $T$. By a small abuse, we keep the notation $h$ as the support function of $T(K)$. This allows us to study the sign of \eqref{eq:J''bis} after a suitable transformation.

Using affine invariance, one can choose $\te_{0}=0$,
and $\te_{1}-\te_{0}=\pi/2$, which ensures that the polygon is contained in a rectangle of
sides $2 h(\theta_0), 2 h(\theta_1)$.  With a further scaling we 
arrange that $h(\theta_0) = h(\theta_1) = 1$ and choose an orientation so that 
$a_1 \le a_0$, see Figure 1.  Under these conditions $A < 4$ (equality would imply the square, excluded 
by hypothesis),
$\tan(\theta_2) < 0$, and a 
trigonometrical calculation shows that

$$
\left|\tan(\theta_2)\right| \ge \frac{2 - a_1}{2 - a_0}.
$$

\begin{figure}[hbt]
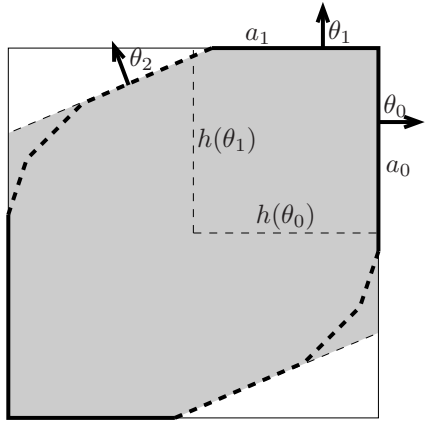

\figinit{0.7pt}
\figpt 0:(0,0)
\figpt 1:(100,100)
\figpt 2:(-100,100)
\figpt 3:(-100,-100)
\figpt 4:(100,-100)

\figpt 5:(100,-10)
\figpt 6:(100,100)
\figpt 7:(10,100)
\figpt 8:(-60,70)
\figpt 9:(-90,40)
\figpt 10:(-100,10)

\figpt 11:(-100,-100)
\figpt 12:(-10,-100)
\figpt 13:(60,-70)
\figpt 14:(90,-40)

\figpt 15:(100,0)
\figpt 16:(0,100)

\figpt 17:(50,0)
\figpt 18:(0,50)
\figpt 19:(35,100)
\figpt 30:(-28,87)
\figpt 31:(100,35)

\figpt 20:(100,60)
\figpt 21:(125,60)
\figpt 22:(70,100)
\figpt 23:(70,125)
\figpt 24:(70,100)
\figpt 25:(70,130)
\figpt 26:(-35,80)
\figpt 27:(-44,104)


\figpt 28:(100,-54)
\figpt 29:(-100,54)

\psbeginfig{}

\psset(fillmode=yes,color=0.8)
\psset(width=0.3)
\psline[1,2,3,4,1]
\psset(fillmode=yes,color=1)
\psline[7,29,2]
\psline[12,28,4]

\psset(fillmode=no,color=0)
\psset(dash=1)
\psset(width=0.3)
\psline[1,2,3,4,1]
\psset(width=1.5)
\psline[5,6,7]
\psline[10,11,12]
\psarrow[20,21]
\psarrow[22,23]
\psarrow[26,27]
\psset(dash=8)
\psline[7,8,9,10]
\psline[12,13,14,5]
\psset(width=0.3)
\psline[0,15]
\psline[0,16]
\psline[7,29]
\psline[12,28]

\psendfig
\figvisu{\figBoxA}{}{
\figwriten 17:{$h(\theta_{0})$}(2)
\figwritee 18:{$h(\theta_{1})$}(2)
\figwriten 19:{$a_{1}$}(2)
\figwritene 20:{$\theta_{0}$}(4)
\figwritene 22:{$\theta_{1}$}(4)
\figwriten 30:{$\theta_{2}$}(1)
\figwritee 31:{$a_{0}$}(4)
}
\centerline{\box\figBoxA}
\caption{
Estimate of \eqref{eq:J''bis} with $\theta_{0}=0, \theta_{1}=\pi/2, h(\te_{0})=h(\te_{1})=1, 0\leq a_{1}\leq a_{0}\leq 2.$
}
\end{figure}

\noindent
Therefore,
\begin{eqnarray*}\label{squarecalc}
-\frac{\sin(\te_{2}-\te_{0})}{\sin(\te_{2}-\te_{1})\sin(\te_{1}-\te_{0})}-4\frac{a_{1}^2}{A(h)}+2
\frac{a_{1}}{h(\te_{1})}
&<& \tan(\theta_{2})- a_1^2 + 2 a_1\\
&\leq& - \frac{2 - a_1}{2 - a_0} - a_1^2 + 2 a_1\\
&=& \frac{2 - a_1}{2 - a_0} \left(a_1(2 - a_0) - 1\right)
\end{eqnarray*}
The factor $\left(a_1(2 - a_0) - 1\right)$ is a harmonic function, negative on the
edges of the triangle $\left\{0 \le a_1 \le a_0 \le 2\right\}$ except 
when $a_1 = a_0 = 1$, where it equals $0$. By the maximum principle it is always 
nonpositive in this triangle.
Observing that the inequality in 
the first line is strict, we conclude that
$J''(h)\cdot(\tilde{v},\tilde{v})<0$.  This contradicts local optimality in the
sense of the $H^1$-distance and concludes the proof of Theorem \ref{th:dim2}.\qed

\begin{remark}
Simple calculations show that this first-order optimality conditions \eqref{11} is satisfied
by any regular symmetric polygon. This explains why we need to analyze the 
second-order condition to get the conclusion.
\end{remark}

\begin{remark}
The invariance of the Mahler functional under affine transformation cannot be simply expressed with the first and second derivatives of $J$, because the support function of $T(K)$ cannot be simply deduced from the support function of $K$. Nevertheless, we can prove that the quantity in \eqref{eq:J''bis} keeps a constant sign under affine transformation.
\end{remark}

\bibliographystyle{plain}

\end{document}